\documentclass[12pt,reqno]{amsart}

\usepackage{graphics,cancel,graphicx}
\usepackage{epsfig,psfrag,manfnt}
\usepackage{bbm,subfigure,rotating,bm,float,mathdots,wasysym,scalerel}
\usepackage{mathrsfs}
\usepackage{graphicx}
\usepackage{stmaryrd}

\usepackage[all,cmtip]{xy}
\usepackage{amssymb,amsmath,amsfonts,amsthm,color}

\usepackage{setspace}
\usepackage[new]{old-arrows}
\usepackage{stackengine}
\stackMath
\newcommand\reallywidehat[1]{%
\savestack{\tmpbox}{\stretchto{%
  \scaleto{%
    \scalerel*[\widthof{\ensuremath{#1}}]{\kern-.6pt\bigwedge\kern-.6pt}%
    {\rule[-\textheight/2]{1ex}{\textheight}}
  }{\textheight}%
}{0.5ex}}%
\stackon[1pt]{#1}{\tmpbox}%
}
\newcommand{\calD}{\mathcal{D}}

\newcommand{\calS}{\mathcal{S}}

\newcommand{\mC}{\mathbb{C}}
\newcommand{\mD}{\mathbb{D}}

\newcommand{\mN}{\mathbb{N}}

\newcommand{\mR}{\mathbb{R}}

\newcommand{\mZ}{\mathbb{Z}}

\newcommand{\bbe}{\bm{e}}

\newcommand{\bbk}{\bm{k}}

\newcommand{\bbm}{\bm{m}}
\newcommand{\bbn}{\bm{n}}

\newcommand{\bbv}{\bm{v}}

\newcommand{\bby}{\bm{y}}

\newtheorem{theorem}{Theorem}[section]

\newtheorem{proposition}[theorem]{Proposition}

\theoremstyle{definition}

\newcommand{\nm}{\,\rule[-.6ex]{.13em}{2.3ex}\,}
\theoremstyle{definition}
\newtheorem{definition}[theorem]{Definition}

\theoremstyle{definition}

\theoremstyle{definition}
\newtheorem{example}[theorem]{Example}

\begin{document}



 \title[]{Ideals in the convolution algebra of periodic distributions}
 
 \author[]{Amol Sasane}
 \address{Department of Mathematics \\London School of Economics\\
     Houghton Street\\ London WC2A 2AE\\ United Kingdom}
 \email{A.J.Sasane@lse.ac.uk}
 
 \maketitle
 
 \begin{abstract} 
 The ring of periodic distributions on $\mR^{\tt d}$ with usual addition and with convolution is considered. Via Fourier series expansions, this ring is isomorphic to the ring $\calS'(\mZ^{\tt d})$ of all maps $f:\mZ^{\tt d}\rightarrow \mC$ of at most polynomial growth (i.e., there exist a real $M>0$ and a nonnegative integer ${\tt m}$ such that for all $\bbn=({\tt n}_1,\cdots, {\tt n}_{\tt d})\in \mZ^{\tt d}$, $ |f(\bbn)|\leq M(1+|{\tt n}_1|+\cdots+|{\tt n}_{\tt d}|)^{\tt m}$), with pointwise operations. 
 It is shown that finitely generated ideals in $\calS'(\mZ^{\tt d})$ are principal, and  ideal membership is 
  characterised analytically. Calling an ideal in $\calS'(\mZ^{\tt d})$ fixed if there is a common index $\bbn \in \mZ^{\tt d}$ where each member vanishes, the fixed maximal ideals are described, and it is shown that not all maximal ideals are fixed. It is shown that finitely generated proper prime ideals in $\calS'(\mZ^{\tt d})$ are fixed maximal ideals. The Krull dimension of $\calS'(\mZ^{\tt d})$ is proved to be infinite, while the weak Krull dimension is shown to be equal to $1$. 
 \end{abstract}
  
\section{Introduction}

\noindent The aim of this article is to study ideals in a naturally arising ring 
in harmonic analysis and distribution theory, namely the ring $\calD'_{\mathbf{V}}(\mR^{\tt d})$ of
periodic distributions with the usual  addition $+$ distributions, and with convolution $\ast$ taken as multiplication. Via a Fourier series
expansion, the ring $(\calD'_{\mathbf{V}}(\mR^{\tt d}),+,\ast)$  is isomorphic to
the ring $\calS'(\mZ^{\tt d})$ consisting of all $f:\mZ^{\tt d} \rightarrow \mC$ of at most polynomial growth,
with pointwise operations, and we recall this below. 

\subsection{The ring  of periodic
  distributions.}
\label{subsec_2}

\noindent For background on periodic distributions and its Fourier
series theory, we refer the reader to \cite[Chapter~16]{Dui}
and \cite[pp.527-529]{Tre}.

Let $\mN=\{1,2,3,\cdots\}$ be the set of natural numbers and $\mZ$ be the set of integers. Consider the space $\calS'(\mZ^{\tt d})$ of all complex valued maps on
$\mZ^{\tt d}$ of at most polynomial growth, that is,
$$
\calS'(\mZ^{\tt d}):=\Big\{f: \mZ^{\tt d}\rightarrow \mC\;\Big|\;
\begin{array}{ll}\exists \textrm{ a real }M>0 \; \exists \,{\tt m}\in \mN\cup \{0\} \;\textrm{ such that} \\
\forall \bbn\in \mZ^{\tt d}, \;
|f(\bbn)|\leq M (1+\nm \bbn\nm)^{\tt m}
\end{array}\Big\},
$$
where $\nm \bbn\nm:=|{\tt n}_1|+\cdots+|{\tt n}_{\tt d}|$ for all
$\bbn=({\tt n}_1,\cdots, {\tt n}_{\;\!\tt d})\in \mZ^{\tt d}$. Then $\calS'(\mZ^{\tt d})$ is a unital
commutative ring with pointwise operations, and the multiplicative
unit element $1_{\mZ^{\tt d}}$ is the constant function $ \mZ^{\tt d}\owns \bbn\mapsto 1$.  The set $\calS'(\mZ^{\tt d})$ equipped with pointwise
operations, is a commutative, unital ring.  Moreover,
$(\calS'(\mZ^{\tt d}),+,\cdot)$ is isomorphic as a ring, to the ring
$(\calD'_{\mathbf{V}}(\mR^{\tt d}), +, \ast)$, where
$\calD'_{\mathbf{V}}(\mR^{\tt d})$ is the set of all periodic distributions 
(with periods described by ${\mathbf{V}}$, see the definition below), with the usual pointwise addition of
distributions, and multiplication taken as convolution of
distributions.

Let $\calD(\mR^d)$ denote the space of  compactly supported infinitely many times differentiable complex valued functions on $\mR^{\tt d}$, and $\calD'(\mR^{\tt d})$ the space of distributions on $\mR^{\tt d}$. 
For $\bbv \in {{\mathbb{R}}}^{\tt d}$, the 
{\em translation operator} ${\mathbf{S}}_{{\bbv}}:\calD'(\mR^{\tt d})\rightarrow \calD'(\mR^{\tt d})$,  is defined by 
 $\langle{\mathbf{S}}_{\bbv}(T),\varphi\rangle = \langle
T,\varphi(\cdot+\bbv)\rangle$ for all $\varphi \in
{\mathcal{D}}({\mathbb{R}}^{\tt d})$. 
A distribution $T\in {\mathcal{D}}'({\mathbb{R}}^{\tt d})$ is called {\em
  periodic with a period}
$\bbv\in {\mathbb{R}}^{\tt d}\setminus \{\mathbf{0}\}$ if
 $T= {\mathbf{S}}_{\bbv}(T)$. 
Let $ {\mathbf{V}}:=\{\bbv_1, \cdots, \bbv_{\tt d}\} $ be a
linearly independent set of ${\tt d}$ vectors in ${\mathbb{R}}^{\tt d}$.  Let 
 ${\mathcal{D}}'_{{\mathbf{V}}}({\mathbb{R}}^{\tt d})$ denote the set
of all distributions $T$ that satisfy
 ${\mathbf{S}}_{\bbv_{\tt k}}(T)=T$ for all  ${\tt k}\in \{1,\cdots, {\tt d}\}$. 
From \cite[\S34]{Don}, $T$ is a tempered distribution, and from the
above it follows by taking Fourier transforms that
$ (1-e^{2\pi i \bbv_{\tt k} \cdot \bby})\widehat{T}=0$, for
$ {\tt k}\in \{1,\cdots, {\tt d}\}$, $\bby \in \mR^d$.  Then 
$
\widehat{T}
=
\sum_{\bbv  \in V^{-1} {\mathbb{Z}}^{\tt d}} 
\alpha_{\bbv}(T) \delta_{\bbv},
$ 
for some scalars $\alpha_{\bbv}(T)\in {\mathbb{C}}$, and where
$V$ is the matrix with its rows equal to the transposes of the column
vectors ${\bbv_1, \cdots, \bbv_{\tt d}}$:
$
V^{\textrm{t}}:= \left[ \begin{smallmatrix} 
    \bbv_1& \cdots & \bbv_{\tt d}
    \end{smallmatrix}\right],
$ with $V^{\textrm{t}}$ denoting the transpose of the matrix $V$. 
Also, in the above, $\delta_{\bbv}$ denotes the usual Dirac
measure with support in $\bbv$, i.e.,  
$
\langle \delta_{\bbv},
\varphi\rangle =\varphi (\bbv)$ for all $ \varphi \in
{\mathcal{D}}({\mathbb{R}}^{\tt d})$.  
Then the Fourier coefficients $\alpha_{\bbv}(T)$ give rise to an
element in $\calS'(\mZ^{\tt d})$, and vice versa, every element in
$\calS'(\mZ^{\tt d})$ is the set of Fourier coefficients of some periodic
distribution. In fact, the ring 
$ (\calD'_{\mathbf{V}}(\mR^{\tt d}),+,\ast) $ of periodic distributions on
$\mR^{\tt d}$ is isomorphic  to the ring $ (\calS'(\mZ^{\tt d}),+,\cdot).  $

In  \cite{Sas}, some algebraic-analytical properties of $(\calS'(\mZ^{\tt d}),+,\cdot)$ were established; see also \cite{RoiSas}. In this article, the structure of  ideals  in this ring is studied, akin to an analogous investigation in \cite{vRen} for a ring of entire functions.

\vspace{-0.15cm}

\subsection{Main results and organisation of the article} 

\vspace{-0.15cm}

\begin{itemize}
\item In \S\ref{section_2}, we  show that finitely generated ideals in $\calS'(\mZ^{\tt d})$ are principal, 
and  ideal membership is characterised analytically. 
\item In \S\ref{section_3}, we describe fixed maximal ideals in $\calS'(\mZ^{\tt d})$, and it is shown that not all maximal ideals are fixed. 
\item In \S\ref{section_4}, we show that finitely generated proper prime ideals in $\calS'(\mZ^{\tt d})$ are fixed maximal ideals. Also, the Krull dimension of  $\calS'(\mZ^{\tt d})$ is proved to be infinite, while the weak Krull dimension is shown to be equal to $1$. 
\end{itemize}

\vspace{-0.45cm}

\section{Finitely generated ideals} 
\label{section_2}

\vspace{-0.15cm}

\begin{proposition}
\label{21_3_2023_1018}
$g$ is a divisor of $f$ in $\calS'(\mZ^{\tt d})$ if and only if there exist a real number $M>0$ and a nonnegative integer ${\tt m}$ such that 
for all $\bbn \in \mZ^{\tt d}$, $|f(\bbn)|\leq M(1+\nm \bbn\nm )^{\tt m} |g(\bbn)|$.
\end{proposition}
\begin{proof} (`If' part:) Define $d:\mZ^{\tt d}\rightarrow \mC$ by 
$$
d(\bbn)=
\left\{\begin{array}{cl} 
\frac{f(\bbn)}{g(\bbn)} &\textrm{if } g(\bbn)\neq 0,\\
0 &\textrm{if } g(\bbn)= 0.
\end{array}\right.
$$
Thus for $g(\bbn)\neq 0$, we have $|d(\bbn)|\leq M(1+\nm \bbn\nm)^{\tt m}$, and this also holds trivially when $g(\bbn)=0$, since the left-hand side is $0$. Thus $d\in \calS'(\mZ^{\tt d})$. Moreover, for $g(\bbn)\neq 0$, we have $d(\bbn)g(\bbn)=f(\bbn)$, and 
when $g(\bbn)=0$, the inequality $ |f(\bbn)|\leq M(1+\nm \bbn\nm )^{\tt m} |g(\bbn)|$ yields $f(\bbn)=0$ too, showing that $d(\bbn)g(\bbn)=d(\bbn)0=0=f(\bbn)$. Hence $dg=f$, as wanted. 

\noindent (`Only if' part:) Suppose that $d\in \calS'(\mZ^{\tt d})$ is such that $dg=f$. Since $d\in \calS'(\mZ^{\tt d})$, there exist $M>0$ and a nonnegative integer ${\tt m}$ such that $|d(\bbn)|\leq M(1+\nm \bbn\nm)^{\tt m}$. So $|f(\bbn)|\!\leq\! |d(\bbn)||g(\bbn)|\!\leq \!M(1\!+\!\nm \bbn\nm)^{\tt m}|g(\bbn)|$ for all $\bbn \in \mZ^{\tt d}$. 
\end{proof}

\noindent In particular,  $f$ is invertible in $\calS'(\mZ^{\tt d})$ if and only if there exists a real number $\delta>0$ and a nonnegative integer ${\tt m}$ such that for all $\bbn\in \mZ^{\tt d}$, 
$|f(\bbn)|\geq \delta (1+|\bbn|)^{-{\tt m}}$. 

\begin{proposition}
\label{prop_17_3_2023_1917}
Every finite number of elements $ f_1,\cdots, f_{\tt K}\in \calS'(\mZ^{\tt d})$ $({\tt K}\in \mN)$ have a greatest common divisor $d$. The element $d$ is given $($up to invertible elements$)$ by $d(\bbn)=\max\{|f_1(\bbn)|,\cdots, |f_{\tt K}(\bbn)|\}$ $(\bbn\in \mZ^{\tt d})$.
\end{proposition}
\begin{proof} Let $d(\bbn)=\max\{|f_1(\bbn)|,\cdots, |f_{\tt K}(\bbn)|\}$ for all $\bbn\in \mZ^{\tt d}$. 
Clearly $d\in \calS'(\mZ^{\tt d})$. As $|f_{\tt k}(\bbn)|\leq |d(\bbn)|$ for all $\bbn \in \mZ^{\tt d}$ and all ${\tt k}\in  \{1,\cdots, {\tt K}\}$, Proposition~\ref{21_3_2023_1018} implies that $d$ is a common divisor of $f_1,\cdots, f_{\tt K}$.  

If $\widetilde{d}\in \calS'(\mZ^{\tt d})$ is a common divisor of $f_1,\cdots, f_{\tt K}$, then by Proposition~\ref{21_3_2023_1018} again, there exist real $M_{\tt k}>0$ and positive integers ${\tt m}_{\tt k}$, for each ${\tt k}\in \{1,\cdots, {\tt K}\}$, such that 
$|f_{\tt k}(\bbn)|\leq M_{\tt k} (1+\nm \bbn \nm)^{{\tt m}_{\tt k}}|\widetilde{d}(\bbn)|$ for all ${\bbn}\in \mZ^{\tt d}$. Setting $M:=\max\{M_1,\cdots, M_{\tt K}\}$ and $m:=\max\{{\tt m}_1,\cdots, {\tt m}_{\tt K}\}$, we get $|d(\bbn)|\leq M(1+\nm \bbn\nm)^{\tt m}|\widetilde{d}(\bbn)|$ for all $\bbn \in \mZ^{\tt d}$. By Proposition~\ref{21_3_2023_1018}, $\widetilde{d}$ divides $d$ in $\calS'(\mZ^{\tt d})$.
\end{proof}

\begin{proposition}
\label{21_3_2023_1937}
Let $\langle f_1,\cdots, f_{\tt K}\rangle$ denote the ideal generated ${\tt K}\in \mN$ elements $ f_1,\cdots, f_{\tt K}\in \calS'(\mZ^{\tt d})$. Then 
$f\in \langle f_1,\cdots, f_{\tt K}\rangle $ if and only if there exists an $M>0$ and a nonnegative integer ${\tt m}$ such that 
$$
|f(\bbn)|\leq M(1+\nm \bbn\nm)^{\tt m} 
\sum_{{\tt k}=1}^{\tt K}|f_{\tt k}(\bbn)| \textrm{ for all }\bbn \in \mZ^{\tt d}.
$$
\end{proposition}
\begin{proof} (`If' part:) For ${\tt k}\in \{1,\cdots, {\tt N}\}$, define 
$g_{\tt k}:\mZ^{\tt d}\rightarrow \mC$ by 
$$
g_{\tt k}(\bbn)=\left\{ \begin{array}{cl}
 \frac{\overline{f_{\tt k}(\bbn)}}{\sum\limits_{\tt j=1}^{\tt K} |f_{\tt j}(\bbn)|^2} f(\bbn)
&\textrm{if } {\scaleobj{0.81}{\sum\limits_{\tt j=1}^{\tt K}}} |f_{\tt j}(\bbn)|^2\neq 0,\\
0 &\textrm{if } {\scaleobj{0.81}{\sum\limits_{\tt j=1}^{\tt K}}} |f_{\tt j}(\bbn)|^2=0.
\end{array}
\right.
$$
If $Q(\bbn):= {\scaleobj{0.81}{\sum\limits_{\tt j=1}^{\tt K} }} |f_{\tt j}(\bbn)|^2\neq 0$, then 
(by Cauchy-Schwarz in the last step), 
$$
\begin{array}{rcl}
|g_{\tt k}(\bbn)|= \frac{|\overline{f_{\tt k}(\bbn)}|}{\sum\limits_{\tt j=1}^{\tt K} |f_{\tt j}(\bbn)|^2} |f(\bbn)|
\!\!\!&\leq &\!\!\!
\frac{\sum\limits_{\tt j=1}^{\tt K} |f_{\tt j}(\bbn)|}{\sum\limits_{\tt j=1}^{\tt K} |f_{\tt j}(\bbn)|^2} M(1\!+\!\nm \bbn\nm)^{\tt m} 
{\scaleobj{0.81}{\sum\limits_{{\tt k}=1}^{\tt K}}}|f_{\tt k}(\bbn)|\\
\!\!\!&\leq &\!\!\! 
\frac{(\sum\limits_{\tt j=1}^{\tt K} |f_{\tt j}(\bbn)|)^2}{\sum\limits_{\tt j=1}^{\tt K} |f_{\tt j}(\bbn)|^2} M(1\!+\!\nm \bbn\nm)^{\tt m} \leq   K M (1\!+\!\nm \bbn\nm)^{\tt m}.
\end{array}
$$
So $g_1,\cdots, g_{\tt K}\in \calS'(\mZ^{\tt d})$. We claim that 
$f_1g_1+\cdots +f_{\tt K} g_{\tt K}=f$. The evaluation of the left-hand side at an  $\bbn\in \mZ^{\tt d}$ such that $ Q(\bbn)\neq 0$ is easily seen to be $f(\bbn)$ by the definition of $g_1,\cdots, g_{\tt K}$. On the other hand, if $Q(\bbn)=0$, then each $f_{\tt k}(\bbn)=0$, and by the given inequality in the statement of the proposition, so is $f(\bbn)=0$. Thus in this case the evaluations at $\bbn$ of both  sides of $f_1g_1+\cdots +f_{\tt K} g_{\tt K}=f$ are zeroes. 

\noindent (`Only if' part:) If $f\in \langle f_1,\cdots, f_{\tt K}\rangle$, then there exist $g_1,\cdots, g_{\tt K}\in \calS'(\mZ^{\tt d})$ such that $f=f_1g_1+\cdots +f_{\tt K} g_{\tt K}$. Let  $M_{\tt k}>0$ and  ${\tt m}_{\tt k}\in \mN\cup\{0\}$, ${\tt k}\in \{1,\cdots, {\tt K}\}$, be such that $|g_{\tt k}(\bbn)|\leq M_{\tt k}(1+\nm \bbn\nm)^{{\tt m}_{\tt k}}$ ($\bbn \in \mZ^{\tt d}$). Then with $M:=\max\{M_1,\cdots, M_{\tt K}\}$ and ${\tt m}:=\max\{{\tt m}_1,\cdots, {\tt m}_{\tt K}\}$, we get 
$$
|f(\bbn)|\leq {\scaleobj{0.81}{\sum_{{\tt k}=1}^{\tt K}}}|f_{\tt k}(\bbn)||g_{\tt k}(\bbn)|
\leq M (1+\nm \bbn\nm)^{\tt m} {\scaleobj{0.81}{\sum_{{\tt k}=1}^{\tt K}}}|f_{\tt k}(\bbn)|.
\eqno\qedhere
$$
\end{proof}

\noindent It follows from Propositions~\ref{prop_17_3_2023_1917} and \ref{21_3_2023_1937} that every finite generated ideal is principal. (Indeed, $\langle f_1,\cdots, f_{\tt K}\rangle =\langle d\rangle$: That $f_{\tt k}\subset \langle d\rangle$ for each ${\tt k}$ is obvious as $d$ divides $f_{\tt k}$, in turn showing  $\langle f_1,\cdots, f_{\tt K}\rangle \subset \langle d\rangle$. For the reverse inclusion, $|d(\bbn)|\!=\!\max\{|f_1(\bbn)|,\cdots, |f_{\tt K}(\bbn)|\}\!\leq\! \sum_{{\tt k}=1}^{\tt K} |f_{\tt k}(\bbn)|$ ($\bbn\!\in \!\mZ^{\tt d}$), and so  by Proposition~\ref{21_3_2023_1937},  $d\in \langle f_1,\cdots, f_{\tt K}\rangle$.  Thus we get  $\langle d\rangle \subset \langle f_1,\cdots, f_{\tt K}\rangle$.)

\vspace{-0.3cm}

\section{Maximal ideals}
\label{section_3}

\begin{definition}
An ideal $\mathfrak{i}$ of $\calS'(\mZ^{\tt d})$ is {\em fixed} if there exists an $\bbk\in \mZ^{\tt d}$ such that for all $f\in \mathfrak{i}$, $f(\bbk)=0$.
\end{definition}

\begin{theorem}
For $\bbk\in \mZ^{\tt d}$, let $\mathfrak{m}_{\bbk}:=\{f\in \calS'(\mZ^{\tt d}): f(\bbk)=0\}$. 
Then $\mathfrak{m}_{\bbk}$ is a fixed maximal ideal of $\calS'(\mZ^{\tt d})$. Every fixed maximal ideal of $\calS'(\mZ^{\tt d})$ is equal to $\mathfrak{m}_{\bbk}$ for some $\bbk\in \mZ^{\tt d}$. 
\end{theorem}
\begin{proof} The fixedness of $\mathfrak{m}_{\bbk}$ is clear. We now show that maximality. As $1_{\mZ^{\tt d}}\in \calS'(\mZ^{\tt d})\setminus \mathfrak{m}_{\bbk}$, $\mathfrak{m}_{\bbk} \subsetneq \calS'(\mZ^{\tt d})$. Let $\mathfrak{i}$ be an ideal such that $\mathfrak{m}_{\bbk}\subsetneq \mathfrak{i}$. Suppose that $f\in \mathfrak{i}\setminus \mathfrak{m}_{\bbk}$. 
Then $f(\bbk)\neq 0$. Define $g\in \calS'(\mZ^{\tt d})$ by 
$g=1_{\mZ^{\tt d}}-\frac{f}{f(\bbk)} $. As $g(\bbk)=0$, we have $g\in \mathfrak{m}_{\bbk}\subset \mathfrak{i}$. Also, $\frac{f}{f(\bbk)} \in \mathfrak{i}$. Thus 
$1_{\mZ^{\tt d}} =g+\frac{f}{f(\bbk)} \in  \mathfrak{i}$, i.e., 
$\mathfrak{i}=\calS'(\mZ^{\tt d})$. 

Next, let $\mathfrak{m}$ be a fixed maximal ideal of $\calS'(\mZ^{\tt d})$. Since $\mathfrak{m}$ is fixed, there exists a $\bbk\in \mZ^{\tt d}$ such that 
$\mathfrak{m}\subset \mathfrak{m}_{\bbk}\subsetneq \calS'(\mZ^{\tt d})$. By the maximality of $\mathfrak{m}$, we conclude that $\mathfrak{m}= \mathfrak{m}_{\bbk}$.
\end{proof}

\begin{example}[Non-fixed maximal ideals]
Let $({\tt k}_{\tt j})_{{\tt j}\in \mN}$ be any subsequence of the sequence of natural numbers. 
Set $\bbk_{\tt j}=({\tt k}_{\tt j},\cdots, {\tt k}_{\tt j})\in \mZ^{\tt d}$. 
Define $
\mathfrak{i}:=\{ f\in \calS'(\mZ^{\tt d}):\lim_{{\tt j}\rightarrow \infty} e^{{\tt k}_{\tt j}} f(\bbk_{\tt j})=0\}.
$ 
 Then $\mathfrak{i}$ is an ideal of $\calS'(\mZ^{\tt d})$. (It is clear that if $f,g\in \mathfrak{i}$, then $f+g\in \mathfrak{i}$. If $f\in \mathfrak{i}$ and $g\in \calS'(\mZ^{\tt d})$, then there exist a real $M>0$ and an ${\tt m}\in \mN\cup\{0\}$ such that 
$|g(\bbn)|\leq M(1+\nm \bbn\nm)^{\tt m}$ for all $\bbn\in \mZ^{\tt d}$, 
 and so 
$$
 |(f g)(\bbk_{\tt j})|
\leq 
|f(\bbk_{\tt j})| M(1+ {\tt d} {\tt k}_{\tt j})^{\tt m}
=
e^{{\tt k}_{\tt j}}|f(\bbk_{\tt j})| e^{-{\tt k}_{\tt j}} M(1+ {\tt d} {\tt k}_{\tt j})^{\tt m}
\stackrel{{\tt j}\rightarrow \infty}{\longrightarrow} 0,
$$ 
showing that $f g\in \calS'(\mZ^{\tt d})$.) Moreover, $\mathfrak{i}\neq \calS'(\mZ^{\tt d})$ since $1_{\mZ^{\tt d}} \not\in \mathfrak{i}$: $e^{{\tt k}_{\tt j}} |1_{\mZ^{\tt d}}(\bbk_{\tt j})|
= e^{{\tt k}_{\tt j}} 1> 1$ for all $n\in \mN$. 
Hence there exists a maximal ideal $\mathfrak{m}$ in $\calS'(\mZ^{\tt d})$ such that $\mathfrak{i}\subset \mathfrak{m}$. We note that for each $\bbk\in \mZ^{\tt d}$, $\mathfrak{m}\neq \mathfrak{m}_{\bbk}$: Define $1_{\{\bbk\}}:\mZ^{\tt d}\rightarrow \mC$  by $1_{\{\bbk\}}(\bbn)=0$ for all $\bbn\neq \bbk$ and $1_{\{\bbk\}}(\bbk)=1$. Then $1_{\{\bbk\}}\in \mathfrak{i} \subset \mathfrak{m}$, but $1_{\{\bbk\}}\not\in \mathfrak{m}_{\bbk}$ as $1_{\{\bbk\}}(\bbk)=1\neq 0$. 
\hfill$\Diamond$
\end{example}

\vspace{-0.45cm}

\section{Prime ideals}
\label{section_4}

\vspace{-0.15cm}

\subsection{Finitely generated proper prime ideals}

\begin{theorem}
\label{19_3_2023_2101}
Let $\mathfrak{p}$ be a finitely generated proper 
prime ideal of ${\mathcal{S}}'(\mZ^{\tt d})$.  
Then there exists an $\bbn_*\in \mZ^{\tt d}$ such that 
 $$
\mathfrak p=\{f\in \calS'(\mZ^{\tt d}): f(\bbn_*)=0\}.
$$  
\end{theorem}

\vspace{-0.6cm}

\begin{proof} We carry out the proof in several steps.

\noindent 
{\bf Step 1. $\mathfrak{p}$ is principal. If $\mathfrak{p}\!=\!\langle d\rangle$, then $d$ has at most one zero.} As $\mathfrak{p}$ is finitely generated, it is principal. Let $d\in \calS'(\mZ^{\tt d})$ be such that $\mathfrak{p}=\langle d \rangle$. 
Let $\bbm,\bbn\in \mZ^{\tt d}$ be distinct and $d(\bbn)=0=d(\bbm)$. Define $a:\mZ^{\tt d}\rightarrow \mC$ by $a(\bbk)=1$ for all $\bbk\neq \bbn$ and $a(\bbn)=0$. 
Then $a\in \calS'(\mZ^{\tt d})$. Also, let $b:\mZ^{\tt d}\rightarrow \mC$ be defined by $b(\bbk)=d(\bbk)$ for all $\bbk\not\in\{ \bbm,\bbn\}$, $b(\bbm)=0$ and $b(\bbn)=1$. Then clearly $b\in \calS'(\mZ^{\tt d})$ too (since it matches with $d$ everywhere except at the single index $\bbn$). Now $(ab)(\bbk)=d(\bbk)$ for all $\bbk\in \mZ^{\tt d}$: If $\bbk \not\in \{\bbm,\bbn\}$ this is clear from the definitions since the left-hand side is $1\cdot d(\bbk)$,  and if $\bbk=\bbm$ or $\bbn$, then both sides are $0$. So $ab=d\in \langle d\rangle=\mathfrak{p}$. But $a\not\in \langle d\rangle$ since otherwise $a=d\tilde{a}$ for some $\tilde{a}\in \calS'(\mZ^{\tt d})$ and then $1=a(\bbm)=d(\bbm) \tilde{a}(\bbm)=0\tilde{a}(\bbm)=0$, a contradiction. Also, $b\not\in \langle d\rangle$ since otherwise $b=d\tilde{b}$ for some $\tilde{b}\in \calS'(\mZ^{\tt d})$ and then $1=b(\bbn)=d(\bbn) \tilde{b}(\bbn)=0\tilde{b}(\bbm)=0$, a contradiction. 
So neither $a$ nor $b$ belong to $\langle d\rangle =\mathfrak{p}$, contradicting the primality of $\mathfrak{p}$. 

\noindent {\bf Step 2.} $\!$Let ${\scaleobj{0.96}{\mathfrak{p}\!=\!\langle d\rangle}}$ (as in Step 1). For each ${\scaleobj{0.96}{\bbn\!\in\! \mZ^{\tt d}}}$, let ${\scaleobj{0.96}{d(\bbn)\!=\!|d(\bbn)|e^{i\theta(\bbn)}}}$ for some $\theta(\bbn)\in (-\pi,\pi]$. Define $h\in \calS'(\mZ^{\tt d})$ by ${\scaleobj{0.96}{h(\bbn)=\sqrt{|d(\bbn)|}e^{i\theta(\bbn)/2}}}$ for all $\bbn\in \mZ^{\tt d}$. 
Then $h^2=d\in \mathfrak{p}$, and as $\mathfrak{p}$ is prime, $h\in \mathfrak{p}$. 

\noindent {\bf Step 3. $d$ has exactly one zero.} We will now show that there exists an $\bbn_*\in \mZ^{\tt d}$ such that $d(\bbn_*)=0$. 
Suppose this is not true. Then by Step~1, $d(\bbn)\neq 0$ for all $\bbn\in \mZ^{\tt d}$. If $h\in \mathfrak{p}$ is as in Step~2, then there exists a $k\in \calS'(\mZ^{\tt d})$ such that $h=kd$, i.e., $\sqrt{|d(\bbn)|}e^{i\theta(\bbn)/2}=|d(\bbn)| e^{i\theta(\bbn)} k(\bbn)$, which yields $1=d(\bbn)(k(\bbn))^2$ ($\bbn\in \mZ^{\tt d}$). Thus $dk^2=1_{\mZ^{\tt d}}$, showing that $d$ is invertible in $\calS'(\mZ^{\tt d})$, contradicting the properness of the ideal $\mathfrak{p}$. 

\noindent {\bf Step 4. We now show that $\mathfrak{p}= \{f\in \calS'(\mZ^{\tt d}):f(\bbn_*)=0\}$.} That $\mathfrak{p}\subset \{f\in \calS'(\mZ^{\tt d}):f(\bbn_*)=0\}$ is clear.  Let $f\in \calS'(\mZ^{\tt d})$ be such that $f(\bbn_*)=0$. Define $g:\mZ^{\tt d}\rightarrow \mC$ by
$$
g(\bbn)=\left\{\begin{array}{cl} 
\frac{f(\bbn)}{d(\bbn)} &\textrm{if } \bbn\neq \bbn_*,\\
0 & \textrm{if } \bbn=\bbn_*.
\end{array}
\right.
$$
Then $f=dg$ (note that $f(\bbn)=d(\bbn)g(\bbn)$ for $\bbn\neq \bbn_*$ follows from the definition of $g$, and $f(\bbn_*)=d(\bbn_*) g(\bbn_*)$ too since both sides are $0$). As the $h$ from Step~2 is in  $\mathfrak{p}$, there exists a $k\in \calS'(Z^{\tt d})$ such that $h=kd$, and so for all $\bbn\in \mZ^{\tt d}$, we get $\sqrt{|d(\bbn)|} e^{i\theta(\bbn)/2} =|d(\bbn)| e^{i\theta(\bbn)} k(\bbn)$, giving $1=|d(\bbn)||k(\bbn)|^2$. 
Hence for $\bbn\neq \bbn_*$, $\frac{1}{|d(\bbn)|}=|k(\bbn)|^2\leq M(1+|\bbn|)^{\tt m}
$ for some $M>0$ and a nonnegative integer ${\tt m}$.  This estimate shows that $g\in \calS'(\mZ^{\tt d})$, and hence $f=gd\in d\calS'(\mZ^{\tt d})=\langle d\rangle=\mathfrak{p}$. 
\end{proof}

\subsection{Krull dimension}

\begin{definition}
The {\em Krull dimension} of a commutative ring $R$ is the 
supremum of the lengths of chains of distinct proper prime ideals of
$R$.
\end{definition}

\noindent Recall that the Hardy algebra $H^\infty$ is the Banach algebra of bounded and holomorphic functions on the unit disc $\mD:=\{z\in \mC: |z|<1\}$, with pointwise operations and the supremum norm $\|\cdot\|_\infty$. In  \cite{vonRen77}, von~Renteln showed that the Krull dimension of 
$H^\infty$ is infinite. We adapt the idea given in \cite{vonRen77}, 
to show that the Krull dimension of $\calS'(\mZ^{\tt d})$ is infinite too. 
A key ingredient of the proof in \cite{vonRen77} was the use of a canonical factorisation of $H^\infty$ elements used to create ideals with zeroes at prescribed locations with prescribed multiplicities. Instead, we will look at the zero set in $ \mZ^{\tt d}$ for $f\in \calS'(\mZ^{\tt d})$,  and use the notion of `zero-order' introduced below. 

If $f\in \calS'(\mZ^{\tt d})$ and $\bbn=({\tt n}_1,\cdots, {\tt n}_{\tt d})\in \mZ^{\tt d}$  is  such that $f(\bbn)=0$, then we define the {\em zero-order $m(f,\bbn)$}  by 
$$
m(f,\bbn)\!=\!\min_{1\leq {\tt k}\leq {\tt d}} 
\max{\scaleobj{1.1}{\Big\{}}{\tt i} \!\in\! \mN: 
\!\!
{\scaleobj{0.9}{\begin{array}{ll} f({\tt n}_1, \cdots, {\tt n}_{{\tt k}-1}, {\tt n}_{\tt k}\!+\!{\tt j}, 
{\tt n}_{{\tt k}+1}, \cdots, {\tt n}_{\tt d})=0 \\
\textrm{ whenever } 0\leq {\tt j} \leq {\tt i}-1
\end{array}}}
\!\!{\scaleobj{1.1}{\Big\}}}.
$$
If $f({\tt n}_1, \cdots, {\tt n}_{{\tt k}-1}, {\tt n}_{\tt k}\!+\!{\tt j}, 
{\tt n}_{{\tt k}+1}, \cdots, {\tt n}_{\tt d})=0$ for all ${\tt j} \in \mN\cup \{0\}$, and all ${\tt k}\in \{1,\cdots, {\tt d}\}$,  then we set $m(f,\bbn)=\infty$. If $f(\bbn)\neq 0$, then we set $m(f,\bbn)=0$. Analogous to the multiplicity of a zero of a (not identically vanishing) holomorphic function, the zero-order satisfies the following property. 
$$ 
\begin{array}{ll}
\textrm{(P1): If }f,g\in \calS'(\mZ^{\tt d})\textrm{ and }\bbn \in \mZ^{\tt d},\\
\phantom{(P1): }\textrm{then }
m(f+g,\bbn)\geq \min\{m(f,\bbn), m(g,\bbn)\}.
\end{array}
$$
The multiplicity of a zero $\zeta$ of the pointwise product of two holomorphic functions is the sum of the multiplicities of $\zeta$ as a zero of each of the two holomorphic functions. For the zero-order, we have the following instead:
$$ 
\begin{array}{ll}
\textrm{(P2): If }f,g\in \calS'(\mZ^{\tt d})\textrm{ and }\bbn \in \mZ^{\tt d},\\
\phantom{(P1): }\textrm{then }m(f g,\bbn)\geq \max\{m(f,\bbn), m(g,\bbn)\}.
\end{array}
$$
We will use the following known result; see \cite[Theorem, \S 0.16, p.6]{GilJer60}.

\begin{proposition}
\label{thm_GilmanJerison}
If $\mathfrak{i}$ is an ideal in a ring $R,$  $M\subset R$ is a set that is closed
under multiplication, and $M\cap \mathfrak{i}=\emptyset$, then there exists an
ideal $\mathfrak{p}$ such that $\mathfrak{i}\subset \mathfrak{p}$ and $\mathfrak{p}\cap M=\emptyset,$ and $\mathfrak{p}$
maximal with respect to these properties. Moreover$,$ such an ideal $\mathfrak{p}$
is necessarily prime.
\end{proposition}

\begin{theorem}
\label{theorem_Krull}
The Krull dimension of ${\mathcal{S}}'(\mZ^{\tt d})$ is infinite.
\end{theorem}
\begin{proof} For ${\tt i}\in \{1,\cdots, {\tt d}\}$, let $\bbe_{\tt i}\in \mZ^{\tt d}$ be the vector all of whose components are zeroes except for the ${\tt i}^{\textrm{th}}$ one, which is defined as $1$.  For ${\tt n}\in \mN$, define $f_{\tt n}\in \calS'(\mZ^{\tt d})$ by 
\begin{eqnarray*}
\!\!\! &&\!\!\! f_{\tt n}(2^{\tt k} \bbe_1 \!+\!{\tt j} \bbe_{\tt i})=0 
\textrm{ if }{\tt k}\in \mN\cup \{0\}, \, 1\!\leq\! {\tt i}\!\leq\! {\tt d},\,
0\!\leq \!{\tt j}\! \leq\! {\tt k}^{{\tt n}+1},
 \\
\!\!\! &&\!\!\!  f_{\tt n}(\bbm) =1\textrm{ if } \bbm\not\in
\{ 2^{\tt k} \bbe_1 \!+\!{\tt j} \bbe_{\tt i}: 
{\tt k}\in \mN\cup \{0\}, \, 1\!\leq \!{\tt i}\!\leq\! {\tt d},\,0\!\leq\! {\tt j} \!\leq\! {\tt k}^{{\tt n}+1} \}.
\end{eqnarray*}
Note that $m(f_{\tt n}, 2^{\tt k} \bbe_1)\geq {\tt k}^{{\tt n}+1}$, but 
for each fixed ${\tt n}\in \mN$, there exists a ${\tt K}_{\tt n}\in \mN\cup \{0\}$ such that  the gap between the indices, 
$$
2^{{\tt k}+1}-2^{\tt k} =2^{\tt k} >{\tt k}^{{\tt n}+1}\textrm{ for all }{\tt k}>{\tt K}_{\tt n},
$$
 and so $m(f_{\tt n},2^{\tt k} \bbe_1)={\tt k}^{{\tt n}+1}$ for all  ${\tt k}>{\tt K}_{\tt n}$. 
Hence 
\begin{equation}
\label{function_family}
\lim\limits_{{\tt k}\rightarrow \infty} \frac{m(f_{\tt n},2^{\tt k} \bbe_1)}{{\tt k}^{\tt n}}=\infty
\quad \textrm{ and } \quad
\lim\limits_{{\tt k}\rightarrow \infty} \frac{m(f_{\tt n},2^{\tt k} \bbe_1)}{{\tt k}^{{\tt n}+1}}=1<\infty.
\end{equation}
Let  
$$
\mathfrak{i}_*:=\{f\in \calS'(\mZ^{\tt d}) : \exists\, {\tt k}_{0}(f) \in \mN_0 \textrm{ such that }
\forall \,{\tt k}>{\tt k}_{0}(f), \; f(2^{\tt k} \bbe_1)=0\}.
$$ 
The set $\mathfrak{i}_*$ is nonempty since $0\in \mathfrak{i}_*$. Clearly $\mathfrak{i}_*$ is closed under
addition, and $f g\in \mathfrak{i}_*$ whenever $f\in \mathfrak{i}_*$ and $g\in \calS'(\mZ^{\tt d})$. So $\mathfrak{i}_*$ is
an ideal of $\calS'(\mZ^{\tt d})$. For ${\tt n}\in \mN$, define
\begin{eqnarray*}
\mathfrak{i}_{\tt n}\!\!\!&=&\!\!\!
\Big\{ f\in \mathfrak{i}_* : \lim\limits_{{\tt k}\rightarrow \infty}
{\scaleobj{0.87}{\frac{m(f,2^{\tt k} \bbe_1)}{{\tt k}^{\tt n}}}}=\infty\Big\},\\
M_{\tt n}\!\!\!&=&\!\!\!
\Big\{ f\in \calS'(\mZ^{\tt d}): \sup\limits_{{\tt k}\in \mN}
{\scaleobj{0.87}{\frac{m(f,2^{\tt k} \bbe_1)}{{\tt k}^{\tt n}}}}<\infty\Big\}. 
\end{eqnarray*}
Clearly $f_{\tt n} \in \mathfrak{i}_{\tt n}$, and so $\mathfrak{i}_{\tt n}$ is not empty. Using
(P1), we see that if $f,g\in \mathfrak{i}_{\tt n}$, then $f+g\in \mathfrak{i}_{\tt n}$.
If $g\in\calS'(\mZ^{\tt d})$ and $f\in \mathfrak{i}_{\tt n}$, then (P2) implies that $f g\in \mathfrak{i}_{\tt n}$. Hence $\mathfrak{i}_{\tt n}$ is an ideal of $\calS'(\mZ^{\tt d})$.

The identity element $1_{\mZ^{\tt d}} \in M_{\tt n}$  for all ${\tt n}\in \mN$. If $f,g \in 
M_{\tt n}$, then it follows from (P2) that $f g\in M_{\tt n}$. Thus
$M_{\tt n}$ is a nonempty multiplicatively closed subset of $\calS'(\mZ^{\tt d})$.

It is easy to check that for all ${\tt n}\in \mN$, $\mathfrak{i}_{{\tt n}+1}\subset \mathfrak{i}_{\tt n}$  and $M_{\tt n}\subset M_{{\tt n}+1}$.  
We now prove that the inclusions are strict for each ${\tt n}\in \mN$. From
\eqref{function_family}, it follows that $f_{\tt n} \in \mathfrak{i}_{\tt n}$ but $f_{\tt n}
\not\in \mathfrak{i}_{{\tt n}+1}$. Also $f_{\tt n} \in M_{{\tt n}+1}$ and $f_{\tt n} \not\in M_{\tt n}$.

Next we show that $\mathfrak{i}_{\tt n} \cap M_{\tt n} =\emptyset$. Indeed, if $f\in
\mathfrak{i}_{\tt n} \cap M_{\tt n}$, then
$$
\infty
= \lim\limits_{{\tt k}\rightarrow \infty} {\scaleobj{0.87}{\frac{m(f,2^{\tt k} \bbe_1)}{{\tt k}^{\tt n}}}}
= \limsup\limits_{{\tt k}\rightarrow \infty}
{\scaleobj{0.87}{ \frac{m(f,2^{\tt k} \bbe_1)}{{\tt k}^{\tt n}}}}
\leq \sup\limits_{{\tt k}\in \mN} {\scaleobj{0.87}{\frac{m(f,2^{\tt k} \bbe_1)}{{\tt k}^{\tt n}}}}
< \infty,
$$
a contradiction. But $\mathfrak{i}_{\tt n} \cap M_{{\tt n}+1} \neq
\emptyset$, since $f_{\tt n} \in \mathfrak{i}_{\tt n} $ and $f_{\tt n}\in M_{{\tt n}+1}$.

We will now show that the Krull dimension of $\calS'(\mZ^{\tt d})$ is infinite by showing that 
for all ${\tt N}\in \mN$, we can construct a chain of strictly decreasing prime ideals 
 $
\mathfrak{p}_{{\tt N}+1}\subsetneq \mathfrak{p}_{\tt N}\subset \cdots \subsetneq \mathfrak{p}_2 \subsetneq \mathfrak{p}_1
$ in $\calS'(\mZ^{\tt d})$.

Fix an ${\tt N}\in \mN$. Applying Proposition~\ref{thm_GilmanJerison} by taking $\mathfrak{i}=\mathfrak{i}_{{\tt N}+1}$ and
$M=M_{{\tt N}+1}$, we obtain the existence of a prime ideal $\mathfrak{p}=\mathfrak{p}_{{\tt N}+1}$ in
$\calS'(\mZ^{\tt d})$, which satisfies $\mathfrak{i}_{{\tt N}+1}
\subset \mathfrak{p}_{{\tt N}+1}$ and $\mathfrak{p}_{{\tt N}+1} \cap M_{{\tt N}+1} =\emptyset$.

We claim the ideal $\mathfrak{i}_{\tt N}+\mathfrak{p}_{{\tt N}+1}$ of $\calS'(\mZ^{\tt d})$ satisfies ${\scaleobj{0.999}{(\mathfrak{i}_{\tt N} +\mathfrak{p}_{{\tt N}+1})
\cap M_{\tt N} = \emptyset}}$. Let $h=f+g \in \mathfrak{i}_{\tt N}+\mathfrak{p}_{{\tt N}+1}$, where $f\in
\mathfrak{i}_{\tt N}$ and $g\in \mathfrak{p}_{{\tt N}+1}$. Since $g\in \mathfrak{p}_{{\tt N}+1}$, by the construction
of ${\mathfrak{p}}_{{\tt N}+1}$ it follows that $g\not\in M_{{\tt N}+1}$. But $M_{\tt N} \subset
M_{{\tt N}+1}$, and so $g\not\in M_{\tt N}$ as well. Thus there exists a
subsequence $({\tt k}_{\tt j})_{\tt j\in \mN}$ of $({\tt k})_{{\tt k}\in \mN}$ such that
$$
\lim_{{\tt j}\rightarrow \infty} 
{\scaleobj{0.87}{\frac{m(g, 2^{{\tt k}_{\tt j}} \bbe_1)}{{\tt k}_{\tt j}^{\tt N}}}}=\infty.
$$
From (P1), we obtain 
$$
{\scaleobj{0.87}{\frac{m(h,   2^{{\tt k}_{\tt j}} \bbe_1 )}{{\tt k}_{\tt j}^{\tt N}}}}
\geq
\min \Big\{
{\scaleobj{0.87}{\frac{m(f, 2^{{\tt k}_{\tt j}} \bbe_1)}{{\tt k}_{\tt j}^{\tt N}}}},
{\scaleobj{0.87}{ \frac{m(g, 2^{{\tt k}_{\tt j}} \bbe_1)}{{\tt k}_{\tt j}^{\tt N}}}}\Big\}.
 $$
  As $f\in \mathfrak{i}_{\tt N}$,  it follows that 
$$
\sup\limits_{{\tt j}\in \mN} 
{\scaleobj{0.87}{\frac{m(h, 2^{{\tt k}_{\tt j}} \bbe_1)}{{\tt k}_{\tt j}^{\tt N}}}}
\geq 
\min \Big\{
\limsup\limits_{{\tt j} \rightarrow \infty} 
{\scaleobj{0.87}{\frac{m(f, 2^{{\tt k}_{\tt j}} \bbe_1)}{{\tt k}_{\tt j}^{\tt N}}}},
\limsup\limits_{{\tt j}\rightarrow \infty} 
{\scaleobj{0.87}{\frac{m(g, 2^{{\tt k}_{\tt j}} \bbe_1)}{{\tt k}_{\tt j}^{\tt N}}}}
\Big\}
\geq \infty.
$$
Thus $h\not\in M_{\tt N}$. Consequently, $(\mathfrak{i}_{\tt N} +\mathfrak{p}_{{\tt N}+1}) \cap M_{\tt N}=\emptyset$.

Clearly $\mathfrak{i}_{\tt N} \subset \mathfrak{i}_{\tt N}+\mathfrak{p}_{{\tt N}+1}$. 
Applying Proposition~\ref{thm_GilmanJerison} again, now taking $\mathfrak{i}=\mathfrak{i}_{\tt N}+\mathfrak{p}_{{\tt N}+1}$ and
$M=M_{\tt N}$, we obtain the existence of a prime ideal $\mathfrak{p}=\mathfrak{p}_{\tt N}$ in
$\calS'(\mZ^{\tt d})$ such that $\mathfrak{i}_{\tt N}+\mathfrak{p}_{{\tt N}+1} \subset \mathfrak{p}_{\tt N}$ and $\mathfrak{p}_{\tt N} \cap M_{\tt N} =\emptyset$. 
Thus $\mathfrak{p}_{{\tt N}+1}\subset \mathfrak{i}_{\tt N}+ \mathfrak{p}_{{\tt N}+1}\subset \mathfrak{p}_{\tt N}$.
The first inclusion is strict as $f_{\tt N}\in \mathfrak{i}_{\tt N}\subset \mathfrak{i}_{\tt N}+\mathfrak{p}_{{\tt N}+1}$. But $f_{\tt N}\not\in \mathfrak{p}_{{\tt N}+1}$ (since $f_{\tt N}\in M_{{\tt N}+1}$ and 
$\mathfrak{p}_{{\tt N}+1}\cap M_{{\tt N}+1}
=\emptyset$ by the construction of $\mathfrak{p}_{{\tt N}+1}$). Thus $\mathfrak{p}_{{\tt N}+1}\subsetneq \mathfrak{p}_{\tt N}$. 

Now consider the ideal $\mathfrak{i}:=\mathfrak{i}_{{\tt N}-1}+\mathfrak{p}_{\tt N}\supset \mathfrak{i}_{{\tt N}-1}$ of $\calS'(\mZ^{\tt d})$ and the multiplicatively closed set $M:=M_{{\tt N}-1}$ of $\calS'(\mZ^{\tt d})$. Similar to the argument given above, we show below that  $\mathfrak{i}\cap M=(\mathfrak{i}_{{\tt N}-1}+\mathfrak{p}_{\tt N})\cap M_{{\tt N}-1}=\emptyset$. 

Let $h=f+g \in \mathfrak{i}_{{\tt N}-1}+\mathfrak{p}_{{\tt N}}$, where $f\in
\mathfrak{i}_{{\tt N}-1}$ and $g\in \mathfrak{p}_{\tt N}$. Since $g\in \mathfrak{p}_{\tt N}$, by the construction
of $\mathfrak{p}_{\tt N}$, $g\not\in M_{\tt N}\supset M_{{\tt N}-1}$, and so $g\not\in M_{{\tt N}-1}$. Thus there exists a subsequence $({\tt k}_{\tt j})_{{\tt j}\in \mN}$ of $({\tt k})_{{\tt k}\in \mN}$ such that
$$
\lim\limits_{{\tt j}\rightarrow \infty} 
{\scaleobj{0.87}{\frac{m(g, 2^{{\tt k}_{\tt j}} \bbe_1)}{{\tt k}_{\tt j}^{{\tt N}-1}}}}=\infty.
$$ 
As $f\in \mathfrak{i}_{{\tt N}-1}$,  
$$
\sup\limits_{{\tt j}\in \mN} 
{\scaleobj{0.87}{\frac{m(h, 2^{{\tt k}_{\tt j}} \bbe_1)}{{\tt k}_{\tt j}^{{\tt N}-1}}}}
\geq 
\min \Big\{
\limsup\limits_{{\tt j} \rightarrow \infty} 
{\scaleobj{0.87}{\frac{m(f, 2^{{\tt k}_{\tt j}} \bbe_1)}{{\tt k}_{\tt j}^{{\tt N}-1}}}},
\limsup\limits_{{\tt j}\rightarrow \infty} 
{\scaleobj{0.87}{\frac{m(g, 2^{{\tt k}_{\tt j}} \bbe_1)}{{\tt k}_{\tt j}^{{\tt N}-1}}}}
\Big\}
\geq \infty.
$$
Thus $h\not\in M_{{\tt N}-1}$. So $(\mathfrak{i}_{{\tt N}-1} +\mathfrak{p}_{\tt N}) \cap M_{{\tt N}-1}=\emptyset$. 

By Proposition~\ref{thm_GilmanJerison}, taking $\mathfrak{i}=\mathfrak{i}_{{\tt N}-1}+\mathfrak{p}_{\tt N}\supset \mathfrak{i}_{{\tt N}-1}$  and $M=M_{{\tt N}-1}$, there exists a prime ideal $\mathfrak{p}=\mathfrak{p}_{{\tt N}-1}$ in $\calS'(\mZ^{\tt d})$ such  that 
$\mathfrak{i}_{{\tt N}-1}+\mathfrak{p}_{\tt N} \subset \mathfrak{p}_{{\tt N}-1}$ and $\mathfrak{p}_{{\tt N}-1} \cap M_{N-1} =\emptyset$. 
Thus $\mathfrak{p}_{\tt N}\subset \mathfrak{i}_{{\tt N}-1}+ \mathfrak{p}_{\tt N}\subset \mathfrak{p}_{{\tt N}-1}$, and again the first inclusion is strict (because $f_{{\tt N}-1}\in \mathfrak{i}_{{\tt N}-1}\subset \mathfrak{i}_{{\tt N}-1}+\mathfrak{p}_{\tt N}$, $f_{{\tt N}-1}\in M_{\tt N}$ and $M_{\tt N}\cap \mathfrak{p}_{\tt N}=\emptyset$). 

Proceeding in this manner, we obtain the chain of distinct prime ideals
 $
\mathfrak{p}_{{\tt N}+1} \subsetneq \mathfrak{p}_{\tt N} \subsetneq \mathfrak{p}_{{\tt N}-1} \subsetneq \cdots \subsetneq \mathfrak{p}_1.
$ in $\calS'(\mZ^{\tt d})$. 
As  ${\tt N}\in \mN$ was arbitrary, it follows that the Krull dimension of $\calS'(\mZ^{\tt d})$ is infinite.
\end{proof}

\subsection{Weak Krull dimension}
Recall the following definition from \cite{Tan}:

\begin{definition}
The {\em weak Krull dimension} of a commutative ring $R$ is the supremum of the lengths of chains of distinct proper finitely generated prime ideals of $R$.
\end{definition}

\begin{theorem}
The weak Krull dimension of ${\mathcal{S}}'(\mZ^{\tt d})$ is $1$. 
\end{theorem}
\begin{proof} 
Let $\mathfrak p_1$ and $\mathfrak p_2$ be finitely generated proper prime ideals in ${\mathcal{S}}'(\mZ^{\tt d})$ such that $\mathfrak p_1\subset \mathfrak p_2$. 
For each ${\tt i}\in \{1,2\}$, by Proposition~\ref{prop_17_3_2023_1917}, there exists an $\bbn_{\tt i}\in \mZ^{\tt d}$ such that $\mathfrak{p}_{\tt i}=\{f\in \calS'(\mZ^{\tt d}):f(\bbn_{\tt i})=0\}$. 
But as $\mathfrak p_1\subset \mathfrak p_2$, it follows that $\bbn_1=\bbn_2$ (by considering the function which is zero at all $\bbn \in \mZ^{\tt d}\setminus \{\bbn_2\}$ and equal to $1$ at $\bbn_2$), and so $\mathfrak{p}_1=\mathfrak{p}_2$. So the weak Krull dimension of $\calS'(\mZ^{\tt d})$ is $1$.
\end{proof}


\begin{thebibliography}{99}


\bibitem{Don}
W. Donoghue, Jr., 
{\em Distributions and Fourier transforms}. 
Pure and Applied Mathematics 32, Academic Press, New York and London, 1969.

\bibitem{Dui} 
J. Duistermaat and  J. Kolk. 
{\em Distributions. Theory and applications.} 
Birkh\"auser, Boston, MA, 2010.

\bibitem{GilJer60}
L. Gillman and M. Jerison.
{\em Rings of continuous functions}.
D. Van Nostrand Company, Princeton, New Jersey, 1960.

\bibitem{RoiSas}
M. Roitman and A. Sasane. 
On the Gleason-Kahane-\.{Z}elazko theorem for associative algebras. 
 {\em Results in Mathematics}, 78:26, no. 1, 2023.

\bibitem{Sas}
A. Sasane. 
A potpourri of algebraic properties of the ring of periodic distributions. 
 {\em Bulletin of the Belgian Mathematical Society. Simon Stevin},
   25:755-776, no. 5, 2018.

\bibitem{Tan}
G. Tang. 
Weak Krull dimension over commutative rings. 
In {\em Advances in ring theory}, 215-224, 
Proceedings of the 4th China-Japan-Korea International Symposium on Ring Theory held in Nanjing, June 24-28, 2004, edited by J. Chen, N. Ding and H. Marubayashi, World Scientific,  2005. 

\bibitem{Tre}
F. Tr{\`e}ves.
{\em Topological vector spaces, Distributions and kernels.}
Unabridged republication of the 1967 original. 
Dover Publications, Mineola, NY, 2006.

\bibitem{vRen} 
M. von Renteln. 
Rings of entire functions with weighted Hadamard multiplication. 
{\em Demonstratio Mathematica},  10:807-813, no. 3-4, 1977. 

\bibitem{vonRen77}
M. von Renteln.
Primeideale in der topologischen algebra $H^{\infty}(\beta)$.
{\em Mathematische Zeitschrift}, 157:79-82, 1977.
  

\end{thebibliography}
\end{document}